\documentclass{amsart}
\usepackage{pinlabel}
\usepackage{amsmath,  amsthm,  amssymb,  amscd,  mathrsfs,  eucal,  epsfig, color,float}
\usepackage{hyperref}
\usepackage{todonotes}
\usepackage{enumerate}
\usepackage{pdfpages}
\usepackage{verbatim}
\usepackage{epstopdf}

\usepackage{tikz}
\usetikzlibrary{decorations.pathreplacing}

\hypersetup{
	unicode=false,          
	pdftoolbar=true,        
	pdfmenubar=true,        
	pdffitwindow=false,     
	pdfstartview={FitH},    
	pdftitle={Ziggurat fringes are self-similar},    
	pdfauthor={Subhadip Chowdhury},     
	pdfsubject={ziggurats, fringes, symmetries},   
	pdfcreator={Subhadip Chowdhury},   
	pdfkeywords={Ziggurat} {Fringe}, 
	pdfnewwindow=true,      
	colorlinks=false,       
	linkcolor=red,          
	citecolor=green,        
	filecolor=magenta,      
	urlcolor=cyan           
}
\usepackage{etoolbox}
\makeatletter
\def\@footnotecolor{blue}
\define@key{Hyp}{footnotecolor}{%
	\HyColor@HyperrefColor{#1}\@footnotecolor%
}
\patchcmd{\@footnotemark}{\hyper@linkstart{link}}{\hyper@linkstart{footnote}}{}{}
\makeatother
\hypersetup{footnotecolor=blue}

\usepackage{microtype}
\DisableLigatures{encoding = *, family = *}

\begin{document}

\newtheorem{theorem}{Theorem}[section]
\newtheorem{lemma}[theorem]{Lemma}
\newtheorem{corollary}[theorem]{Corollary}
\newtheorem{conjecture}[theorem]{Conjecture}
\newtheorem{proposition}[theorem]{Proposition}
\newtheorem{question}[theorem]{Question}
\newtheorem{problem}[theorem]{Problem}
\newtheorem*{fringe_formula}{Fringe Formula~\ref{theorem:fringe_formula}}
\newtheorem*{sigma inequality}{$\sigma$-inequality~\ref{theorem:sigma_inequality}}
\newtheorem*{self_similar}{Projective Self Similarity~\ref{theorem:self_similarity}}
\newtheorem*{claim}{Claim}
\newtheorem*{criterion}{Criterion}
\theoremstyle{definition}
\newtheorem{definition}[theorem]{Definition}
\newtheorem{construction}[theorem]{Construction}
\newtheorem{notation}[theorem]{Notation}
\newtheorem{convention}[theorem]{Convention}
\newtheorem*{warning}{Warning}

\theoremstyle{remark}
\newtheorem{remark}[theorem]{Remark}
\newtheorem{example}[theorem]{Example}
\newtheorem{scholium}[theorem]{Scholium}
\newtheorem*{case}{Case}

\newcommand\id{\textnormal{id}}
\newcommand\Z{\mathbb Z}
\newcommand\R{\mathbb R}
\newcommand{\scl}{\textnormal{scl}}
\newcommand{\genus}{\textnormal{genus}}
\newcommand{\homeo}{\textnormal{Homeo}}
\newcommand{\rot}{\textnormal{rot}}
\newcommand{\Hom}{\textnormal{Hom}}
\newcommand{\SL}{\textnormal{SL}}
\newcommand{\fr}{\textnormal{fr}}
\newcommand{\vf}[1]{\frac{1}{#1}}

\title{Ziggurat fringes are self-similar}
\author{Subhadip Chowdhury}
\address{Department of Mathematics \\ University of Chicago \\
Chicago,  Illinois,  60637}
\email{subhadip@math.uchicago.edu}
\date{\today}

\begin{abstract}
We give explicit formulae for fringe lengths of the Calegari-Walker {\em ziggurats} --
 i.e.\/ graphs of extremal rotation numbers associated to positive words in
free groups. These formulae reveal (partial) integral projective self-similarity in ziggurat fringes, which are low-dimensional projections of characteristic polyhedra on the bounded cohomology of free groups. This explains phenomena observed experimentally by Gordenko and Calegari-Walker.
\end{abstract}

\maketitle
\setcounter{tocdepth}{1}
\tableofcontents

\section{Introduction}

Let $\homeo_+^\sim(S^1)$ denote the group of homeomorphisms of the real line
that commute with integer translation,  and let
$\rot^\sim:\homeo_+^\sim(S^1) \to \R$ denote Poincar\'e's (real-valued)
rotation number. Let $F$ be a free group on two generators $a, b$ and let
$w$ be a word in the semigroup generated by $a$ and $b$ (such a $w \in F$ is
said to be {\em positive}). Let $h_a(w)$ and
$h_b(w)$ be the number of $a$'s and $b$'s respectively in $w$. The {\em fringe}
associated to $w$ and a rational number $0\le p/q < 1$ is the set of
$0\le t < 1$ for which there is a homomorphism from $F$ to
$\homeo_+^\sim(S^1)$ with $\rot^\sim(a)=p/q$,  $\rot^\sim(b)=t$ and
$\rot^\sim(w) = h_a(w)p/q + h_b(w)$. Calegari-Walker show that there is some
least rational number $s \in [0, 1)$ so that the fringe associated to $w$ and
to $p/q$ is equal to the interval $[s, 1)$. The {\em fringe length},  denoted
$\fr_w(p/q)$,  is equal to $1-s$.

The main theorem we prove in this paper is an explicit formula for fringe
length:

\begin{fringe_formula}
	 If $w$ is positive,  and $p/q$ is a reduced fraction, then
	 $$\fr_w(p/q) = \frac {1} {\sigma_w(g)\cdot q}$$
	 where $\sigma_w(g)$ depends on the word $w$ and on $g:=gcd(q, h_a(w))$. Furthermore, $g\cdot\sigma_w(g)$ is an integer.
\end{fringe_formula}

As $t\to 1$, the dynamics of $F$ on $S^1$ is approximated better and better by a linear model. For $t$ close to $1$, the nonlinearity can be characterized by a perturbative model; fringes are the maximal regions where this perturbative model is valid. Our main theorem says that the size of this region of stability follows a power law. This is a new example of (topological) nonlinear phase locking in $1$-dimensional dynamics giving rise to a power law, of which the most famous example is the phenomenon of Arnol'd Tongues \cite{RJD}.

\subsection{Motivation}
If $G$ is a Lie group,  and $\Gamma$ is a finitely generated group,  one studies
representations of $\Gamma$ into $G$ up to conjugacy not by looking at the
quotient space $\Hom(\Gamma, G)/G$ (which is usually non-Hausdorff),  but by taking
a further (maximal) quotient on which certain natural functions -- {\em characters}
-- are continuous and well-defined; i.e.\/ one studies {\em character varieties}. 

Recovering a representation from a character is not always straightforward.
Given a (finite) subset $S$ of $\Gamma$, it becomes an interesting and subtle
question to ask what constraints are satisfied by the values of a character on $S$.
For example, the (multiplicative) Horn problem poses the problem of
determining the possible values of the spectrum of the product AB
of two unitary matrices given the spectra of A and B individually.
There is a map
$$\Lambda: SU(n) \times SU(n) \to \R^{3n}$$
taking $A$, $B$ to the logarithms of the spectra of $A$, $B$ and $AB$
(suitably normalized). Agnihotri-Woodward \cite{Agnihotri_Woodward}  and Belkale \cite{Belkale} proved that the
image is a convex polytope, and explicitly described the image.

When $G$ is replaced with a topological group such as $\homeo_+^\sim(S^1)$ (the
group of orientation-preserving homeomorphisms of the circle),
the situation becomes more complicated. Recall that the (real-valued)
rotation number
\[\rot^\sim:\homeo_+^\sim(S^1) \to \R\]
is constant on conjugacy classes (more precisely, on semi-conjugacy classes; see
e.g.\/ Ghys \cite{Ghys} or Bucher-Frigerio-Hartnick \cite{Bucher_Frigerio_Hartnick}, see section \ref{subsection:rotation_number} for more details)  and can be thought of
as the analog of a character in this context. Following Calegari-Walker \cite{Calegari_Walker}
we would like to understand what constraints are simultaneously satisfied
by the value of $\rot^\sim$ on the image of a finite subset of $\Gamma$
under a homomorphism to $\homeo_+^\sim(S^1)$. I.e.\/ we study the
values $x_i:=\rot^\sim(\rho(w_i))$ for finitely many $w_i\in \Gamma$ on a
common representation $\rho$.

\subsection{Free Groups, Positive words and Ziggurats}
The universal case to understand is that of a free group. Thus, 
let $F$ be a free group with generators $a$,  $b$,  and for any element
$w\in F$ let $x_w$ be the function from conjugacy classes of representations
$\rho:F \to \homeo_+^\sim(S^1)$ to $\R$ which sends a representation $\rho$ to
$x_w(\rho):=\rot^\sim(\rho(w))$. The $x_w$ are {\em coordinates} on the space of conjugacy classes of representations,  and we study this space
through its projections to finite dimensional spaces obtained from finitely many
of these coordinates.

For any $w \in F$ and for any $r, s\in\R$ we can define
$$X(w;r, s) = \lbrace x_w(\rho) \; | \; x_a(\rho)=r,  x_b(\rho)=s \rbrace$$
Then $X(w;r, s)$ is a {\em compact} interval (i.e.\/ the extrema are achieved)
and it satisfies $X(w;r+m, s+n) = X(w;r, s)+mh_a(w)+nh_b(w)$ where $h_a, h_b:F \to \Z$
count the signed number of copies of $a$ and $b$ respectively in each word.

If we define $R(w;r, s) = \max \{X(w;r, s)\}$ then $\min \{X(w;r, s)\} = -R(w;-r, -s)$. So
all the information about $X(w;r, s)$ can be recovered from the function
$R(w;\cdot, \cdot):\R^2 \to \R$. In fact,  by the observations made above,  it
suffices to restrict the domain of $R$ to the unit square $[0, 1)\times[0, 1)$.

The theory developed in  \cite{Calegari_Walker} is most useful
when $w$ is a {\em positive word}; i.e.\/ a word in the {\em semigroup} generated
by $a$ and $b$. In this case,  $R(w;r, s)$ is lower semi-continuous,  and 
monotone non-decreasing in both its arguments. Furthermore it is {\em locally
constant} and takes {\em rational values} on an open and dense subset of
$\R^2$. In fact, 

\begin{theorem}[Calegari-Walker \cite{Calegari_Walker} Thm. 3.4, 3.7]\label{theorem:Calegari_Walker_rational}
Suppose $w$ is positive (and not a power of $a$ or $b$),  and suppose $r$ and
$s$ are rational. Then
\begin{enumerate}
\item{$R(w;r, s)$ is rational with denominator no bigger than the smaller
of the denominators of $r$ and $s$; and}
\item{there is some $\epsilon(r, s)>0$ so that $R(w;\cdot, \cdot)$ is constant
on $[r, r+\epsilon)\times [s, s\times \epsilon)$.}
\end{enumerate}
\end{theorem}

Furthermore,  when $r$ and $s$ are rational and $w$ is positive,  Calegari-Walker
give an explicit combinatorial algorithm to compute $R(w;r, s)$; it is the
existence and properties of this algorithm that proves 
Theorem~\ref{theorem:Calegari_Walker_rational}. Computer implementation of this
algorithm allows one to draw pictures of the graph of $R$ (restricted to
$[0, 1)\times[0, 1)$) for certain short words $w$,  producing a stairstep structure
dubbed a {\em Ziggurat}; see Figure~\ref{ziggurat}. 

\begin{figure}[htpb]
\labellist
\small\hair 2pt
\endlabellist
\centering
\includegraphics[width = \textwidth]{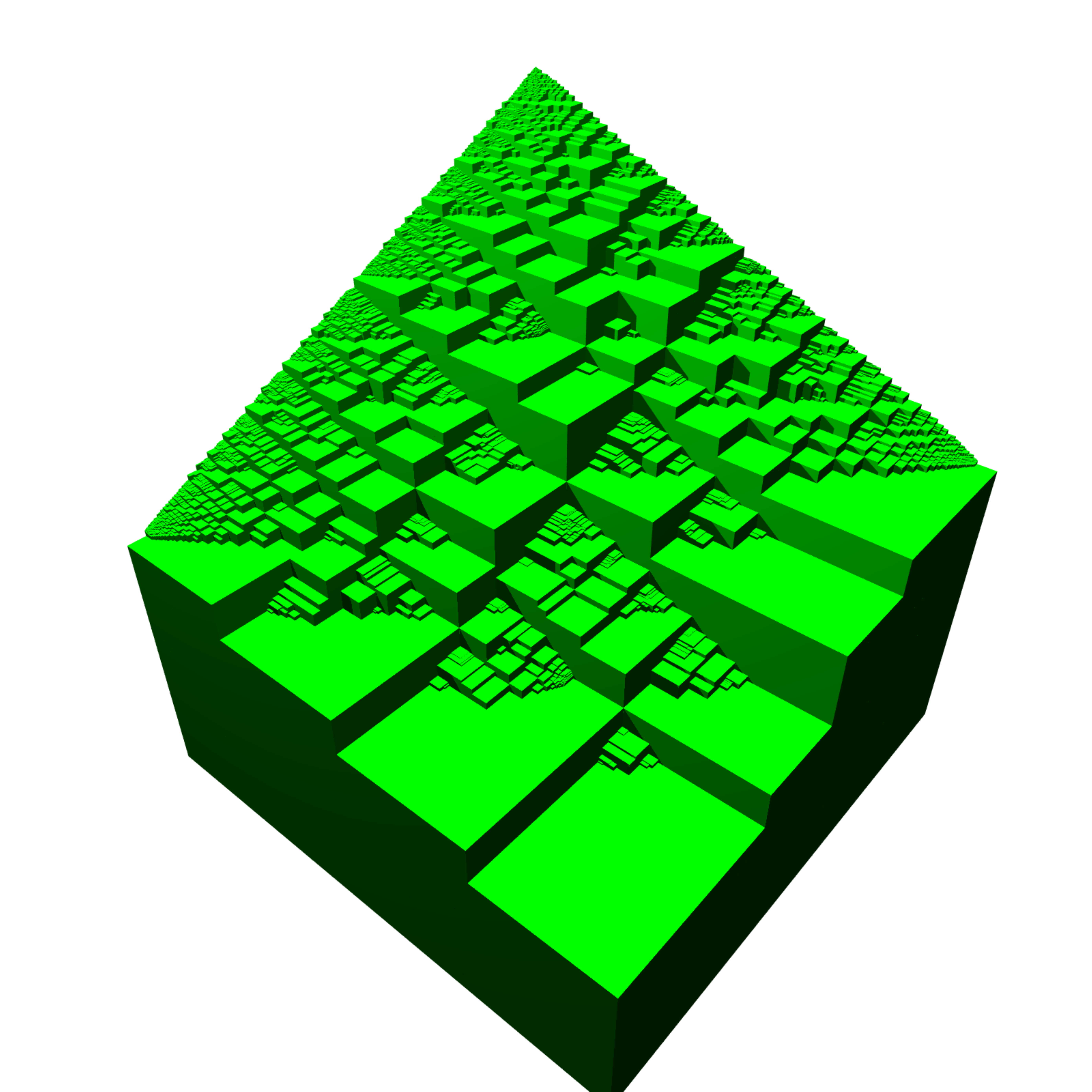}
\caption{Graph of $R(abbbabaaaabbabb;\cdot, \cdot)$; colloquially,  a {\em ziggurat}. Picture courtesy of
Calegari-Walker.}\label{ziggurat}
\end{figure}

In the special case of the word $w=ab$,  a complete analysis
can be made,  and an explicit formula obtained for $R(ab;\cdot, \cdot)$ (this case
arose earlier in the context of the classification of taut foliations of
Seifert fibered spaces,  where the formula was conjectured by Jankins-Neumann
\cite{Jankins_Neumann} and proved by Naimi \cite{Naimi}). But in
{\em no other case} is any explicit formula known or even conjectured,  
and even the computation of $R(w;r, s)$ takes time exponential 
in the denominators of $r$ and $s$.

\subsection{Projective self-similarity and fringes}

In a recent preprint,  Gordenko \cite{Gordenko} gave a new analysis and interpretation
of the $ab$ formula,  relating it to the Naimi formula in an unexpected
way. Her formulation exhibits and explains an {\em integral projective self-similarity}
of the $ab$-ziggurat,  related to the theory of continued fractions,  and the
fact that the automorphism group of $F_2$ is $\SL(2, \Z)$. Such global self-similarity
is (unfortunately) not evident in ziggurats associated to other positive words;
but there is a partial self-similarity (observed experimentally by 
Calegari-Walker and by Gordenko) in the {\em germ} of the ziggurats near
the {\em fringes} where one of the coordinates $r$ or $s$ approaches $1$ from
below.

If we fix a positive word $w$ and a rational number $r$,  and (following
\cite{Calegari_Walker}) we denote by $R(w;r, 1-)$ the limit of $R(w;r, t)$ as $t \to 1$
from below,  then the following can be proved:

\begin{theorem}[Calegari-Walker \cite{Calegari_Walker} Prop. 3.15]\label{theorem:Calegari_Walker_fringe}
If $w$ is positive,  and $r$ is rational,  there is a least rational number
$s \in [0, 1)$ so that $R(w;r, t)$ is constant on the interval $[s, 1)$ and equal to
$h_a(w)r + h_b(w)$.
\end{theorem}

We refer to the number $1-s$ as in Theorem~\ref{theorem:Calegari_Walker_fringe} 
(depending on the word $w$ and the rational number $r$)
as the {\em fringe length} of $r$,  and denote it $\fr_w(r)$,  or just by $\fr(r)$
if $w$ is understood. In other words,  $\fr_w(r)$ is the greatest number such that
$R(w;r, 1-\fr_w(r)) = h_a(w)r+ h_b(w)$. More precisely,  we should call this a ``left fringe'', 
where the right fringe should be the analog with the roles of the generators
$a$ and $b$ interchanged.

\subsection{Statement of results}

\S~\ref{section:background} summarizes background,  including some elements
from the theory of ziggurats from \cite{Calegari_Walker}. The most important
ingredient is a description of the Stairstep Algorithm.

In \S~\ref{section:fringe} we undertake an analysis of the Stairstep Algorithm
when applied to the computation of fringe lengths. A number of remarkable 
simplifications emerge which allows us to reduce the analysis to a tractable
combinatorial problem which depends (in a complicated way) only on 
$\gcd(q, h_a(w))$.

Our main theorem gives an explicit formula for $\fr_w$ for
any positive word $w$,  and establishes a (partial) integral projective self-similarity for 
fringes, thus giving a theoretical basis for the experimental observations 
of Calegari-Walker and Gordenko.

\begin{fringe_formula}
If $w$ is positive,  and $p/q$ is a reduced fraction,  then
$$\fr_w(p/q) = \frac {1} {\sigma_w(g) \cdot q}$$
where $\sigma_w(g)$ depends only on the word $w$ and $g:=\gcd(q, h_a(w))$;  and $g\cdot\sigma(q)$ is an
integer.
\end{fringe_formula}

The function $\sigma_w(g)$ depends on $w$ and on $q$ in a complicated way,  but
there are some special cases which are easier to understand. In \S~\ref{section:examples} we prove the following inequality:

\begin{sigma inequality}
	Suppose $w=a^{\alpha_1}b^{\beta_1}a^{\alpha_2}b^{\beta_2}\ldots a^{\alpha_n}b^{\beta_n}$. Then the function $\sigma_w(g)$ satisfies the inequality
	$$\frac {h_b(w)}{h_a(w)} \le \sigma_w(g) \le \max \beta_i$$
	Moreover,  $h_b(w)/h_a(w) = \sigma_w(g)$ when $h_a$ divides $q$,  and
	$\sigma_w(g) = \max \beta_i$ when $q$ and $h_a(w)$ are coprime.

\end{sigma inequality}

The Fringe Formula explains the fact that $\fr_w(p/q)$ is independent of $p$ (for $\gcd(p, q)=1)$
and implies a periodicity of $\fr_w$ on infinitely many scales. More precise
statements are found in \S~\ref{section:Projective_Self_Similarity}.
\subsection{Acknowledgement}
I would like to thank Anna Gordenko for sharing her preprint \cite{Gordenko} which directly
inspired the main problem studied in this paper,  and to Victor Kleptsyn, Alden Walker and Jonathan Bowden for some useful discussions. I would also like to thank Clark Butler, Amie Wilkinson and Paul Apisa for several helpful comments. Finally I would like to thank Danny Calegari, my advisor, for his continued support and guidance, as well as for the extensive comments and corrections on this paper, and for providing the thanksgiving turkey.
\section{Background}\label{section:background}

\subsection{Rotation numbers} \label{subsection:rotation_number}

Consider the central extension
$$0 \to \Z \to \homeo_+^\sim(S^1) \to \homeo_+(S^1) \to 0$$
whose center is generated by unit translation $z:p \to p+1$.

Poincar\'e defined the {\em rotation number} $\rot:\homeo_+(S^1) \to \R/\Z$ as
follows. First,  define a function
$\rot^\sim:\homeo_+^\sim(S^1) \to \R$ by
$$\rot^\sim(g) = \lim_{n \to \infty} \frac {g^n(0)} {n}$$
Then $\rot^\sim(gz^n) = \rot^\sim(g) + n$ for any integer $n$,  so that
$\rot^\sim$ descends to a well-defined function $\rot:\homeo_+(S^1) \to \R/\Z$.

Recall that for $F$ a free group generated by $a$,  $b$,  for any $w\in F$ and
for any numbers $r, s\in \R$ we define $R(w;r, s)$ to be the maximum value of
$\rot^\sim(\rho(w))$ under all homomorphisms $\rho:F \to \homeo_+^\sim(S^1)$ for
which $\rot^\sim(\rho(a))=r$ and $\rot^\sim(\rho(b))=s$. The
maximum is achieved on some representation $\rho$ for any fixed $r$ and $s$ (Calegari-Walker \cite{Calegari_Walker}, Lemma $2.13$),  
but the function $R(w;\cdot, \cdot)$ is typically not continuous in either 
$r$ or $s$.

\subsection{Positive words and $XY$ words}

Now suppose $w$ is a positive word (i.e. containing only positive powers of $a$ and $b$),  and $r=p_1/q_1$,  $s=p_2/q_2$ are rational and
expressed in reduced form. Theorem~\ref{theorem:Calegari_Walker_rational} says that
$R(w;p_1/q_1, p_2/q_2)$ is rational,  with denominator no bigger than $\min(q_1, q_2)$. Following \cite{Calegari_Walker}, we
present the Calegari-Walker algorithm to compute $R(w;p_1/q_1, p_2/q_2)$ using purely combinatorial means.

\begin{definition}[$XY$-word]\label{definition:XY_word}
An $XY$-word of {\em type $(q_1, q_2)$} is a cyclic word in the
2-letter alphabet $X, Y$ of length $q_1+q_2$,  with a total of $q_1$ $X$'s and $q_2$ $Y$'s.
\end{definition}

If $W$ is an $XY$-word of type $(q_1, q_2)$,  we let $W^\infty$ denote the bi-infinite
string obtained by concatenating $W$ infinitely many times,  and think of this bi-infinite
word as a function from $\Z$ to $\lbrace X, Y\rbrace$; we denote the image of $i \in \Z$ under
this function by $W_i$,  so that each $W_i$ is an $X$ or a $Y$, 
and $W_{i+q_1+q_2} = W_i$ for any $i$. 

We define an action of the {\em semigroup} generated by
$a$ and $b$ on $\Z$,  associated to the word $W$ (see Figure \ref{figure:action_a_b_w}). The action is given as follows. For each integer
$i$,  we define $a(i) = j$ where $j$ is the least index such that the sequence 
$W_i,  W_{i+1},  \cdots ,  W_j$ contains exactly $p_1+1$ $X$'s.
Similarly,  $b(i) = j$ where $j$ is the least index such that the sequence 
$W_i,  W_{i+1},  \cdots ,  W_j$ contains exactly $p_2+1$ $Y$'s. Note that this means $W_{a(i)}$ is always an $X$ and respectively $W_{b(i)}$ is always $Y$. We can then define
\[\rot_W^\sim(w) = \lim_{n \to \infty} \frac {w^n(1)} {n\cdot(q_1+q_2) }.\]

\begin{figure}[!ht]
	\centering
	\begin{tikzpicture}[scale=0.5, place/.style={circle,  draw=blue!50,   thick,  inner sep=0pt,  minimum size=6mm}]
	\node at (-6, 0) {$Y$};
	\node at (-5, 0) {$X$};
	\node at (-4, 0) {$X$};
	\node (v1) at (-3, 0) {$Y$};
	\node at (-2, 0) {$X$};
	\node at (-1, 0) {$Y$};
	\node at (0, 0) {$X$};
	\node (v2) at (1, 0) {$X$};
	\node at (2, 0) {$Y$};
	\node at (3, 0) {$X$};
	\node (v3) at (4, 0) {$Y$};
	\node at (5, 0) {$X$};
	\node at (6, 0) {$X$};
	\node at (7, 0) {$X$};
	\node (v4) at (8, 0) {$X$};
	\node at (9, 0) {$X$};
	\node at (10, 0) {$Y$};
	\node at (11, 0) {$X$};
	\node at (12, 0) {$X$};
	\node (v5) at (13, 0) {$Y$};
	\path node (c1) at (v1) [place] {}
	node (c2) at (v2) [place] {}
	node (c3) at (v3) [place] {}
	node (c4) at (v4) [place] {}
	node (c5) at (v5) [place] {};
	\node at (-4, 1) {} edge [->, bend left] node[above] {$b$} (c1);
	\node at (-2, 1) {} edge [<-, bend right] node[above] {$a$} (c1);
	\node at (0, 1) {} edge [->, bend left] node[above] {$a$} (c2);
	\node at (2, 1) {} edge [<-, bend right] node[above] {$a$} (c2);
	\node at (3, 1) {} edge [->, bend left] node[above] {$b$} (c3);
	\node at (5, 1) {} edge [<-, bend right] node[above] {$a$} (c3);
	\node at (7, 1) {} edge [->, bend left] node[above] {$a$} (c4);
	\node at (9, 1) {} edge [<-, bend right] node[above] {$b$} (c4);
	\node at (12, 1) {} edge [->, bend left] node[above] {$b$} (c5);
	\node at (14, 1) {} edge [<-, bend right] node[above] {$b$} (c5);
	\end{tikzpicture}
	\caption{Action of $a$ and $b$ on $W$}
	\label{figure:action_a_b_w}
\end{figure}
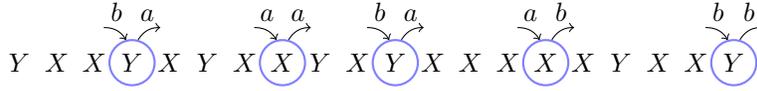
\begin{proposition}[Calegari-Walker formula]\label{proposition:XY_formula}
With notation as above,  there is a formula
\[R(w;p_1/q_1, p_2/q_2) = \max_W \{\rot_W^\sim(w)\}\]
where the maximum is taken over the {\em finite} set of $XY$-words $W$ of type $(q_1, q_2)$.
\end{proposition}
Evidently,  each $\rot_W^\sim(w)$ is rational,  with denominator less than or equal to
$\min(q_1, q_2)$,  proving the first part of Theorem~\ref{theorem:Calegari_Walker_rational}.
Though theoretically interesting,  a serious practical 
drawback of this proposition is that the number
of $XY$-words of type $(q_1, q_2)$ grows exponentially in the $q_i$. 

\subsection{Stairstep Algorithm}\label{subsection:stairstep_algorithm}

In this subsection we discuss the Stairstep algorithm found in \cite{Calegari_Walker} in more details and in the context of this paper.

\begin{theorem}[Calegari-Walker \cite{Calegari_Walker},  Thm. 3.11]\label{theorem:stairstep}
Let $w$ be a positive word,  and suppose $p/q$ and $c/d$ are rational numbers so that
$c/d$ is a value of $R(w;p/q, \cdot)$. Then 
\[u: = \inf \lbrace t \; : \; R(w;p/q, t) = c/d \rbrace\]
is rational,  and $R(w;p/q, u)=c/d$.
\end{theorem}

The theorem is proved by giving an algorithm (the Stairstep Algorithm)
to compute $u$ and analyzing its properties. 
Note that the fringe length $\fr_w(p/q)$ is the value of $1-u$ where $u$
is the output of the Stairstep Algorithm for $c/d = h_a(w)p/q + h_b(w)$. Observe that, whereas Theorem $\ref{theorem:Calegari_Walker_fringe}$ proved the existence of a fringe length, this theorem proves that the length is in fact a rational number.
We now explain this algorithm.
\begin{proof}

Since $R$ is monotone non-decreasing in both of its arguments,  it suffices to
prove that
\begin{equation}
\inf \lbrace t\; : \; R(w;p/q, t)\ge c/d\rbrace \label{stairstep_inequality}
\end{equation} 
is rational,  and the infimum is achieved. Also,  since $R$ is locally constant
from the right at rational points,  it suffices to compute the infimum over
rational $t$. So consider some $t=u/v$ (in lowest terms) such that $R(w;p/q, u/v)\geq c/d$. In fact,  let $W$ be a $XY$ word of type $(q, v)$ for which $R(w;p/q, u/v)=\rot^\sim_W(w)$. After some cyclic permutation,  we can write $$W=Y^{t_1}XY^{t_2}XY^{t_3}X\ldots Y^{t_q}X$$ where $t_i\geq 0$ and $\sum_{i=1}^{q}t_i=v$. Our goal is then to minimize $u/v$ over all such possible $XY$-words $W$.

After some circular permutation (which does not affect $R$),  we may also assume without loss of generality that $w$ is of the form $$w=b^{\beta_n}a^{\alpha_n}\cdots b^{\beta_2} a^{\alpha_2} b^{\beta_1}a^{\alpha_1}$$ where $\alpha_i, \beta_i>0$. Also,  assume that equality is achieved in (\ref{stairstep_inequality}) for $u/v$. Thus by construction,  the action of $w$ on $W$,  defined via its action on $\Z$,  is periodic with a period $d$,  and a typical periodic orbit begins at $W_1=Y$.
 
We fix some notations and try to analyze the action of each maximal string of $a$ or $b$ in $w$ on $W$ by inspecting its action on $\Z$. Note that,  for \[\tilde{s}_i= a^{\alpha_{i}} b^{\beta_{i-1}} a^{\alpha_{i-1}} \cdots b^{\beta_1} a^{\alpha_1}(1), \] the $\tilde{s}_i$'{th} letter in $W^{\infty}$ is always $X$. Let $s_i$ be the index modulo $q$ so that $W^\infty_{\tilde{s_i}}$ is the $s_i$'th $X$ in $W$(cf. Figure \ref{figure:si_th_X}). Thus for a periodic orbit starting at $W_1=Y$,  the string $b^{\beta_i}$ is applied to the $s_i$'{th} $X$.
\begin{figure}[!ht]
	\centering
	\begin{tikzpicture}[place/.style={circle,  draw=blue!50,   thick,  inner sep=0pt,  minimum size=7mm}]
	\node at (-1, -5) {$Y^{t_1}$};
	\node (x1) at (0, -5) {$X$};
	\node at (1, -5) {$Y^{t_2}$};
	\node (x2) at (2, -5) {$X$};
	\node at (3, -5) {$\cdots$};
	\node at (4, -5) {$Y^{t_{s_i}}$};
	\node (xi) at (5, -5) {$X$};
	\node at (6, -5) {$Y^{t_{s_i}+1}$};
	\node at (7, -5) {$\cdots$};
	\node at (8, -5) {$Y^{t_q}$};
	\node (xq) at (9, -5) {$X$};
	\path node (x1c) at (x1)[place] {};
	\path node (x2c) at (x2)[place] {};
	\path node (xic) at (xi)[place] {};
	\path node (xqc) at (xq)[place] {};
	\node (x1label) at (0, -7) {$1^{st}$} edge [red, ->] (x1c);
	\node (x2label) at (2, -7) {$2^{nd}$} edge [red, ->] (x2c);
	\node (xilabel) at (5, -7) {$s_i^{th}$} edge [red, ->] (xic);
	\node (xqlabel) at (9, -7) {$q^{th}$} edge [red, ->] (xqc);
	\end{tikzpicture}
	\caption{The $XY$ word of type $(q,v)$.}
	\label{figure:si_th_X}
\end{figure}
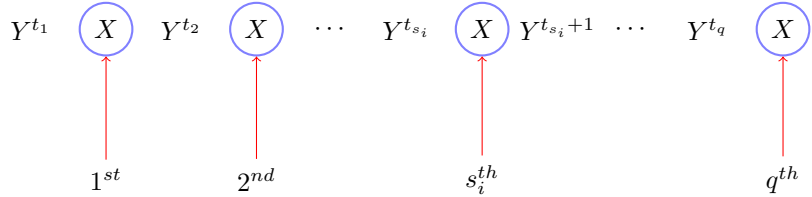

Then by definition,  $b^{\beta_i}(\tilde{s}_i)$ is the least number such that  the sequence $W_{\tilde{s}_i},$ $W_{\tilde{s}_i+1}, \cdots, W_{b^{\beta_i}(\tilde{s}_i)}$ contains exactly $u\beta_i + 1$ $Y$'s. Let $l_i$ denote the number of $X$'s in the sequence $W_{\tilde{s}_i}(=X), W_{\tilde{s}_i+1}, \cdots, W_{b^{\beta_i}(\tilde{s}_i)}(=Y)$(cf. Figure \ref{figure:action_of_b^beta_i}). Thus  $l_i$ is the smallest number such that 
\begin{equation}
t_{s_i+1}+t_{s_i+2}+\ldots + t_{s_i+l_i+1}\geq u\beta_i +1 \label{l_i_inequality}
\end{equation} 
In other words,  $l_i$ is the biggest number such that 
\begin{equation}
t_{s_i+1}+t_{s_i+2}+\ldots + t_{s_i+l_i}\leq u\beta_i 
\end{equation} 
\begin{figure}[!ht]
	\centering
	\begin{tikzpicture}[place/.style={circle,  draw=blue!50,   thick,  inner sep=0pt,  minimum size=14mm}, scale=0.78]
	\node at (-10, -5) {$\cdots$};
	\node at (-9, -5) {$Y^{t_*}$};
	\node at (-8, -5) {$X$};
	\node at (-7, -5) {$Y^{t_{*}}$};
	\node at (-6, -5) {$\cdots$};
	\node (xi) at (-4.8, -5) {$X$};
	\node (yi1)at (-3, -5) {$Y^{t_{s_i}+1}$};
	\node (yij) at (-1.5, -5) {$\ldots$};
	\node (yil) at (0, -5) {$Y^{t_{s_i+l_i}}$};
	\node (xil) at (1.5, -5) {$X$};
	\node (yil1) at (3, -5) {$Y^{t_{s_i+l_i+1}}$};
	\node at (5, -5) {$\cdots$};
	\path node (cix) at (xi)[place] {}
	node(cily) at (yil1)[place]{};
	\path (cix) edge[->, bend left] node [above]{$b^{\beta_i}$} (cily) ;
	\node (xilabel) at (-4.8, -7) {$s_i$'th } edge [red, ->] (xi);
	\node (xillabel) at (1.5, -7) {$(s_i + l_i)$'th } edge [red, ->] (xil);
	\node (totalY) at (-1.5,-8) {Total no. of $Y$ is $\leq u\beta_i$};
	\path (totalY) edge[red,->] (yi1);
	\path (totalY) edge[red,->] (yil);
	\path (totalY) edge[red,->] (yij);
	\end{tikzpicture}
	\caption{Action of $b^{\beta_i}$}
	\label{figure:action_of_b^beta_i}
\end{figure}
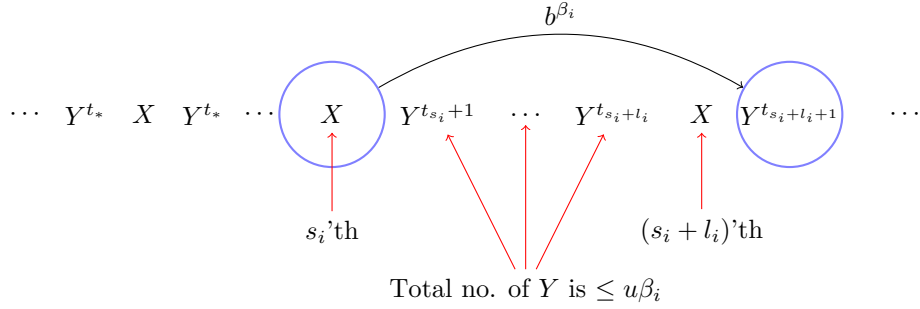

The purpose of rewriting this inequality was to make it homogeneous. Even if equality does not occur in (\ref{stairstep_inequality}), the inequality in (\ref{l_i_inequality}) still holds true. The only difference is that $l_i$ does not necessarily have to be the smallest number,  however it does have to satisfy other constraints which we now describe.

We write $w^d$ as \[w^d=b^{\beta_k}a^{\alpha_k}b^{\beta_{k-1}}a^{\alpha_{k-1}}\cdots b^{\beta_1}a^{\alpha_1}\] and instead of considering the action of $w$ on $W$ with a period $d$,  assume that $w^d$ acts on $W^c$ by its action on $\Z$. Then the maximal $a-$strings and $b-$strings in $w^d$,  all together cover exactly the total number of $X$'s (and $Y$'s) in $W^c$.  For a similar reason,  we know that intervals of the form of $\left(W_j,W_{a^{\alpha_i}(j)}\right)$ enclose precisely $p\alpha_i + 1$ $X$'s.  Thus we get the equality
\[\sum_{i=1}^k (l_i + (\alpha_ip + 1)) = cq.\]
Note that here $\alpha_i$'s are periodic as a function of $i$,  with a period $k/d=n$,  but in general,  the $l_i$'s are not periodic in $i$. We can also give a formula for $s_i$ by counting the number of $X$'s covered.
\[s_i=\sum_{j=1}^{i}(\alpha_j p+1) + \sum_{j=1}^{i-1}l_j.\]

Thus,  we have formulated our minimization problem as a homogeneous linear integral equation subject to finitely many integral linear constraints. Because of homogeneity,  it has a solution in integers if and only if it has a solution in rational numbers,  and consequently,  we can normalize the whole problem by rescaling to $v=1$. Our algorithm is then as follows:

\begin{enumerate}[Step 1.]
\item Replacing $w$ by a cyclic permutation if necessary,  write $w^d$ in the form $w^d=b^{\beta_k}a^{\alpha_k}\ldots b^{\beta_1}a^{\alpha_1} $.
\item Enumerate all non-negative integral solutions to \[ \sum_{i=1}^{k}l_i= cq- \sum_{i=1}^{k}(\alpha_i p +1). \] 
\item For each such solution set $(l_1, \ldots, l_k)$,  define \[s_i=\sum_{j=1}^{i}(\alpha_j p+1) + \sum_{j=1}^{i-1}l_j\]
\item Find the smallest $u$ which satisfies the system of inequalities
\[\begin{cases}
\sum\limits_{i=1}^{q}t_i =1,  \\
 t_i \geq 0 \, \, \,  \forall \,  i, \\
 t_{s_i+1}+t_{s_i+2}+\ldots+t_{s_i + l_i}  \leq u\beta_i & \forall \,  1\leq i\leq k  \, ( \text{indices taken}\!\!\!\mod q )
\end{cases} \]
\item Find the smallest $u$ over all solution sets $(l_1, \ldots, l_k)$.
\end{enumerate}

The solution to this algorithm is necessarily rational and gives the minimal $t$ such that $R(w;p/q, t)\geq c/d$. Also if equality is achieved then clearly $R(w;p/q, u)=c/d$,  and thus the theorem is proved.
\end{proof}

\section{A formula for fringe lengths}\label{section:fringe}

In this section we will apply the Stairstep Algorithm to the computation of fringe lengths.
The key idea is that in this special case,  the equation 
$$\sum_{i=1}^k l_i = cq - \sum_{i=1}^k (\alpha_ip + 1)$$
has a {\em unique} non-negative integral solution. This in turn reduces the last step of the
algorithm to the solution of a {\em single} linear programming problem,  
rather than a system of (exponentially) many inequalities.

\subsection{Statement of Fringe Formula}

First let us state the Fringe Formula. 

\begin{theorem}[Fringe Formula]\label{theorem:fringe_formula}
If $w$ is positive,  and $p/q$ is a reduced fraction,  then
$$\fr_w(p/q) = \frac {1} {\sigma_w(g) \cdot q}$$
where $\sigma_w(g)$ depends only on the word $w$ and $g:=\gcd(q, h_a(w))$;  and $g\cdot\sigma_w(g)$ is an
integer.
\end{theorem}

The formula for $\sigma_w(g)$ depends on both the $\alpha_i$ and the $\beta_j$ in a complicated way, 
which we will explain in the sequel.

\subsection{Proof of the Fringe Formula}

We now begin the proof of the Fringe Formula. This takes several steps,  and requires a careful
analysis of the Stairstep Algorithm. We therefore adhere to the notation in \S~\ref{subsection:stairstep_algorithm}. After cyclically permuting $w$ if necessary we write $w$ in the form
\[w=b^{\beta_n}a^{\alpha_n}\ldots b^{\beta_1}a^{\alpha_1}.\]

\subsubsection{Finding the optimal partition}

First note that by Theorem (\ref{theorem:Calegari_Walker_fringe}),  it is enough to find the minimum $t$ such that 
\[R(w;p/q, t)=\frac{h_ap+h_bq}{q}.\] Thus to apply the stairstep algorithm (\ref{theorem:stairstep}),  we are going to fix $c/d=(h_ap+h_bq)/q$ where $c/d$ is the reduced form. Let us denote the $gcd$ of $h_a$ and $q$ by $g$ so that we have
\[c=\frac{h_ap+h_bq}{g}, \,  d=\frac{q}{g}\]
since $(p, q)=1$. Further writing $h_a=h'g$ and $q=q'g$,  we rewrite the above equations as
\[c=h'p+h_bq', \,  d=q'. \]
Thus step $1$ of our algorithm becomes
\[w^{q'}=b^{\beta_{nq'}}a^{\alpha_{nq'}}\ldots b^{\beta_1}a^{\alpha_1}\] where clearly $\alpha_i, \beta_i$ are periodic as functions of $i$ with period $n$. 

Similarly, step $2$ of our algorithm transforms to
\[l_1+\ldots+l_{q'.n}=\underbrace{\frac{h_a.p + h_b.q}{g}}_{=c}.q - \underbrace{q'.h_a}_{=\sum_{i=1}^{nq'}\alpha_i}.p - q'.n \] i.e. 
\begin{align}
l_1+\ldots+l_{nq'}= {h_b. qq'} -nq'
\end{align} 
and the equations in step $4$ to find the minimum solution $u$,  become
\begin{align}
\sum\limits_{i=1}^{q}t_i &=1 \label{geneqn1} \\
t_i &\geq 0 &\forall \,  i\label{geneqn2}\\
t_{s_i+1}+t_{s_i+2}+\ldots+t_{s_i + l_i} & \leq \beta_iu &\forall \,  1\leq i\leq nq' \label{geneqn3}
\end{align} where indices are taken$\!\pmod{q}$.
Now if any of the $l_i$ is greater than or equal to $q\beta_i$,  then the indices on the LHS of equation (\ref{geneqn3}) cycle through all of $1$ through $q$ a total of $\beta_i$ times. Then using (\ref{geneqn1}),  we get that
\[\beta_i=\beta_i\sum_{1}^{q}t_i\leq t_{s_i+1}+t_{s_i+2}+\ldots+t_{s_i + l_i} \leq \beta_iu\] implying $u\geq 1$,  which is clearly not the optimal solution. Hence for the minimal solution $u$,  we must have
\[l_i\leq q\beta_i-1, \,  \forall\,  1\leq i\leq nq'.\]
Summing up all of these inequalities,  we get that
\[\sum_{i=1}^{nq'}l_i\leq q\sum_{i=1}^{nq'}\beta_i - nq' = qq'h_b -nq'\] But on the other hand,  by step $2$,  equality is indeed achieved in the inequality above  and hence 
\begin{equation}
l_i=q\beta_i-1, \,  \forall\,  1\leq i\leq nq'
\end{equation} 
is the \emph{unique} non-negative integral solution to the partition problem in step $2$. As mentioned before, this means we only need to deal with a single linear programming problem henceforth,  formulated more precisely in the next section.

\subsubsection{A linear programming problem} 
With the specific values of $l_i$ found above,  we can transform equations (\ref{geneqn1}),  (\ref{geneqn2}) and (\ref{geneqn3}) as follows. Note that for $l_i=q\beta_i-1$,  the set of indices ${s_i+1}, {s_i+2}, \cdots, {s_i+l_i}$ cycle through all of the values $1, 2, \cdots,  q$ a total of $\beta_i$ times,  except one of them,  namely $s_i\pmod{q}$,  which appears $\beta_i-1$ times. Then we can rewrite (\ref{geneqn3}) as 
\[\beta_i\left(\sum_{j=1}^{q}t_j\right)-t_{s_i}\leq \beta_i u \quad \forall \,  1\leq i\leq nq'\] i.e.
\[\frac{t_{s_i}}{\beta_i} \geq 1-u \quad \forall \,  1\leq i\leq nq'\]

Observe that in the above equation,  $\beta_i$'s are periodic with a period $n$ whereas the $s_i$'s are well defined modulo $q$ (since $t_i$'s have period $q$),  which is usually much bigger than $n$. Then for the purpose of finding an $u$ which satisfies the system of equations (\ref{geneqn1}),  (\ref{geneqn2}) and (\ref{geneqn3}), it will be enough to consider the indices $i$ for which $\beta_i$ is maximum for the same value of $s_i$. 

To make the statement more precise,  we introduce the following notation. Let  the set of indices $\Lambda$ be defined by
\[\Lambda=\left\{i\, \left|\,  \beta_i=\max\limits_{\substack{s_j=s_i \\ 1\leq j\leq nq'}}\beta_j  \right.\right\}.\] Then the first thing to note is that the set of numbers $\{s_i\}_{i\in \Lambda}$ are all distinct. Next recall that we are in fact trying to find the fringe length,  which is $1-t$,  where $t$ is the solution to the stairstep algorithm. So with a simple change of variable,  our algorithm becomes the following \emph{linear programming problem}:
\begin{align*}
\text{Find maximum of}\quad &\min_{i\in \Lambda}\left\{\vf{\beta_i} t_{s_i}\right\} \\
\text{Subject to}\quad &\sum_{i\in \Lambda}t_{s_i}\leq 1, \,  t_{s_i}\geq 0 \,  \forall i
\end{align*}

But since we are trying to find the maximum,  we may as well assume that $\sum_{i\in \Lambda}t_{s_i}=1$ and $t_k=0$ if $k\neq s_i$ for some $i\in \Lambda$. Then by a theorem of Kaplan\cite{Kaplan},  we get that the optimal solution occurs when for all $i\in \Lambda$,  the number ${t_{s_i}}/{\beta_i}$ equals some constant $T$ independent of $i$. To find $T$,  observe that
\[\frac{t_{s_i}}{\beta_i}=T \Rightarrow \sum_{i\in \Lambda}\beta_i T = 1 \Rightarrow T= \vf{\sum_{i\in \Lambda}\beta_i}. \] Thus the optimal solution to the linear programming problem,  which is also the required fringe length is given by
\begin{equation}
{\fr_w(p/q)=\vf{\sum_{i\in \Lambda}\beta_i}.}
\end{equation}
So all that remains is to figure out what the set of indices $\Lambda$ looks like. In the rest of this section we try to characterize $\Lambda$ and prove the fringe formula \ref{theorem:fringe_formula}.

\subsubsection{Reduction to combinatorics}
It is clear from the definition that to figure out the set $\Lambda$,  we need to find out exactly when two of the $s_i$'s are equal as $i$ ranges from $1$ to $nq'$. Recall that the indices $s_i$ are taken modulo $q$. Using the optimal partition,  we get that 
\[s_i+l_i=\sum_{j=1}^{i}(p\alpha_j+1+q\beta_j-1)\] and hence \[s_I=s_{J} \Leftrightarrow \sum_{j=1}^{I}\alpha_j\equiv \sum_{j=1}^{J}\alpha_j \pmod{q}\] since $l_I\equiv l_{J}\pmod{q}$. Thus the elements of $\Lambda$ are in bijective correspondence with the number of residue classes modulo $q$ in the following set of numbers:
\begin{align*}
A_1&=\alpha_1\\
A_2&=\alpha_1+\alpha_2\\
A_3&=\alpha_1+\alpha_2+\alpha_3\\
A_4&=\alpha_1+\alpha_2+\alpha_3+\alpha_4\\
\vdots&\\
A_{nq'}&=\alpha_1+\alpha_2+\alpha_3+\ldots+\alpha_{nq'}
\end{align*}
So we can rewrite the formula for the set $\Lambda$ as \[\Lambda=\left\{ i \left| \beta_i=\max\limits_{\substack{A_j\equiv A_i\pmod{q}\\ 1\leq j\leq nq'}} \beta_j \right.\right\}\]
Note that $A_n=h_a$ and $\alpha_i$'s are periodic with period $n$. So we have,  $A_{n+i}=A_i+h_a$ or in other words,  the collection of numbers $A_1, A_2, \ldots, A_{nq'}$ is nothing but a union of disjoint translates of the collection $(A_1, A_2, \ldots, A_n)$ by $0, h_a, 2h_a, \ldots, (q'-1)h_a$.

Let us refer to the $n$-tuple $(A_1, A_2, \ldots, A_n)$ as the first ``$n$-block". Similarly the $h_a$-translate of the first $n$-block is referred to as the second $n$-block and so on. Note that $q'h_a=h'q$,  so the $q'h_a$-translate of the first $n$-block is identical to itself modulo $q$. Hence we may think of translation by $(q'-1)h_a$ as translation by $-h_a$.

Next we claim that 
\begin{claim}
The numbers $0, h_a, 2h_a, \ldots, (q'-1)h_a$ are all distinct modulo $q$.
\end{claim}
\begin{proof}If $q$ divides the difference between any two such numbers,  say $mh_a$,  then $q'\mid mh' \Rightarrow q'\mid m \Rightarrow m\geq q'$,  which is a contradiction. 
\end{proof}

In fact since $h'$ is invertible modulo $q$, the set of numbers $\{0, h_a, \ldots, (q'-1)h_a\}$ is the same as $\{0, g, 2g, \ldots, (q'-1)g\}$ modulo $q$. Thus to determine the congruence classes in the collection $A_1, A_2, \ldots, A_{nq'}$,  it is enough to find out which $n$-blocks overlap with the first $n-$block. Note that translating an $n-$block by $h_a(=h'g)$ takes it off itself entirely, so the \emph{only} translates of an $n$-block that could overlap with itself are the translates by $ig$ for $|i|<h'$ (See Figure \ref{fig:n-blocks}).

\begin{figure}[H]
	\centering
	\def\svgwidth{\textwidth}
	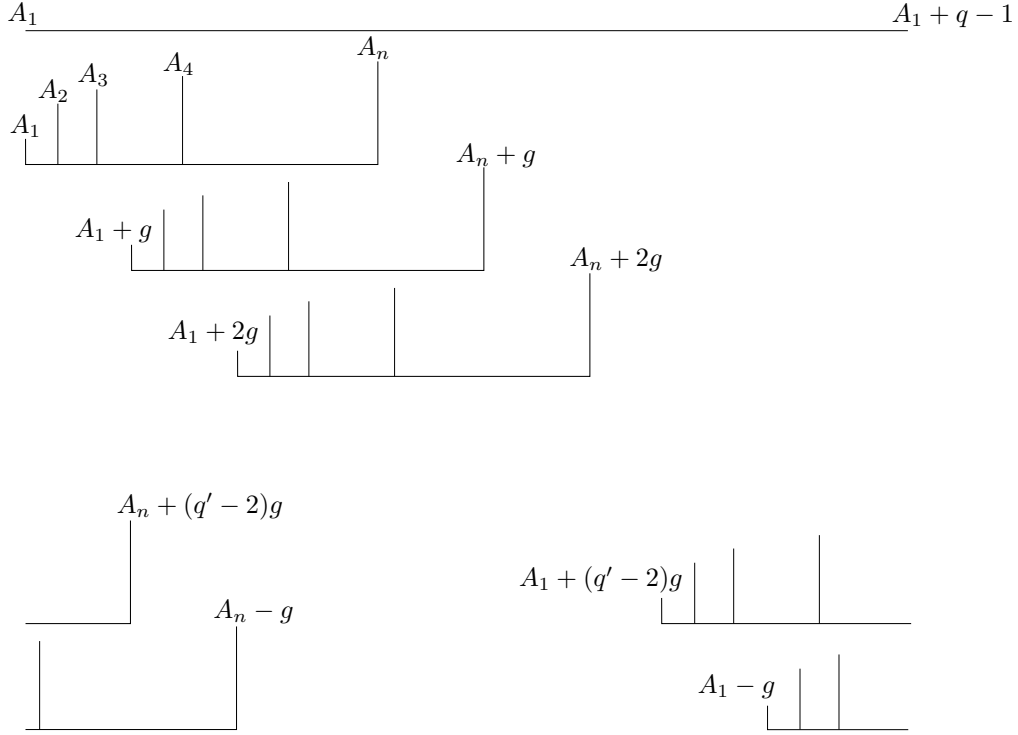
	\caption{Translates of the first $n$-block}
	\label{fig:n-blocks}
\end{figure}
Finally observe that if we start with the the $n$-block given by $(A_1+g, A_{2}+g, \ldots, A_{n}+g)$ instead,  we get overlaps at the same multiple of $g$ as the first $n$-block; only translated by $g$. Thus starting from $A_1$,  if we divide the residue class of $q$ into a total of $q'$ number of $g-$sized groups,  then each $\beta_i$'s appears same no. of times in each group and the overlaps appear at the same places translated by multiples of $g$. Hence to calculate the sum of $\max\{\beta_i\}$ over all residue classes,  it is enough to calculate it for the residue classes which appear among $A_1, A_1+1,  A_1+2, \ldots$ up to $A_1+(g-1)$ and then multiply the result by $q'$.

Let us summarize the results we have found so far in the form of an algorithm.
\begin{enumerate}[Step 1.]
\item Write down $A_1, A_2, \ldots, A_n$ where $A_i=\alpha_1+\ldots+\alpha_i$. 
\item For each $0\leq i\leq g-1$, let $\mathfrak{B}_i$ be defined as follows:
\[\mathfrak{B}_i=\max\left\{\beta_{k+mg}\,\,|\,\, A_{k+mg}\equiv A_1+i\pmod{q}\text{ where } -h'<m<h', 1\leq k\leq n\right\}\] Note that in case $q'<h'$, we replace $h'$ with $q'$ in above definition.
\item Let $S$ be the sum of $\mathfrak{B}_i$'s for $0\leq i\leq g-1$. Then the Fringe length is given by 
\begin{equation}
\fr_w(p/q)=\vf{q'S}
\end{equation}
\end{enumerate}

To finish the proof,  define $\sigma_w(g):=S/g$ and note that by the structure of the algorithm,  $\sigma_w(g)$ depends only on $g=gcd(q, h_a)$ and the word $w$. As a corollary,  we also get the remarkable consequence that
\begin{corollary}
The Fringe length \emph{does not} depend on $p$.
\end{corollary} 
i.e. the Fringes are ``periodic'' on every scale. In section \S~\ref{section:Projective_Self_Similarity} we elaborate on this phenomenon in a particular example, and discuss possible generalizations.

\section{Examples and special cases}\label{section:examples}

In this section we give some examples to illustrate the complexity of the function $\sigma$ in general, 
and in the special case that $h_a(w)$ is prime. Let us first prove that
\begin{theorem}[$\sigma$-inequality]
	\label{theorem:sigma_inequality}
	Suppose $w=a^{\alpha_1}b^{\beta_1}a^{\alpha_2}b^{\beta_2}\ldots a^{\alpha_n}b^{\beta_n}$. Then the function  $\sigma_w(g)$ satisfies the inequality
	\[ \frac{h_b}{h_a} \leq \sigma_w(g) \leq \max\limits_{1\leq i\leq n} \beta_i\] where the first equality is achieved in the case when $h_a$ divides $q$ and the second equality occurs when $(q, h_a)=1$.
\end{theorem}   
\begin{proof}
For the first inequality,  recall the numbers $A_1, A_2, \ldots,  A_{nq'}$ from last section. Note that the fact that $h_a\cdot q'=h'\cdot q$ tells us that there are at most $h'$ elements in each residue class modulo $q$ among $A_1, \ldots, A_{nq'}$. Thus 
\[\sum_{i=1}^{nq'}t_{s_i}\leq h'\cdot \sum_{i\in \Lambda}t_{s_i} \leq h'\cdot \sum_{i=1}^{q}t_i =h'\]
On the other hand,  adding all the $nq'$ inequalities in (\ref{geneqn3}),  and using $l_i=q\beta_i-1$,  we get that
\[u.\sum_{i=1}^{nq'}\beta_i \geq \sum_{i=1}^{nq'}\left(\beta_i\sum_{j=1}^{q}t_j - t_{s_i}\right)=\sum_{i=1}^{nq'}\beta_i - \sum_{i=1}^{nq'}t_{s_i}\geq \sum_{i=1}^{nq'}\beta_i - h' \] 
\[u\geq 1- \frac{h'}{h_b.q'}= 1-\frac{h_a}{h_bq} \]
Hence,  for the minimal $u$ giving the fringe length we get that
\[\sigma_w(g)\geq \frac{h_b}{h_a}.\]

For the second inequality,  observe that by definition,  \[\fr_w(p/q)=\vf{\sigma_w(g)q}=\vf{\sum_{i\in \Lambda}\beta_i}\geq \vf{|\Lambda|\cdot \max_{i\in\Lambda}\beta_i } \geq \vf{q\cdot \max_{i\in\Lambda}\beta_i }\] since number of elements in $\Lambda$ is at most the number of residue classes modulo $q$. Hence \[\sigma_w(g)\leq \max_{i\in\Lambda}\beta_i \leq \max_{1\leq i\leq n}\beta_i.\]
We will finish the proof by showing that equality is indeed achieved in the following special cases:
\subsection*{Case 1 : $h_a \mid q$} In this case $h'=1$. Hence all the $s_i$'s are distinct.

Consider the specific example where $t_{s_i}={\beta_i}/({h_bq'})$ for all $i$ and the rest of the $t_i$'s are zero. Then we have 
\[\beta_i.u\geq \sum_{j\neq i} \frac{\beta_j}{h_bq'}.\beta_i + \frac{\beta_i}{h_bq'}.(\beta_i-1)=\beta_i.\frac{h_bq'}{h_bq'}-\frac{\beta_i}{h_bq'} \Rightarrow u\geq 1-\frac{1}{h_bq'}. \]
Thus the minimum $u_0$ which gives a solution to (\ref{geneqn1}), (\ref{geneqn2}), (\ref{geneqn3}) is $1-1/(q'h_b)= 1-{h_a}/({h_b q}).$ Thus equality is achieved in the first part of Theorem \ref{theorem:sigma_inequality}.

We can give a second proof of this same fact using the algorithm developed in last section. Since $h_a\mid q$,  the $gcd$ of $h_a$ and $q$ is $h_a$. So any $g-$translate of the $n$-block is disjoint from itself. Hence $S=h_b$,  giving the same formula as above.

\subsection*{Case 2: $gcd(h_a, q)=1$} In this situation,  $g=1$.  Hence $c=h_a.p+h_b.q$ and $d=q$ since $q=q'$. 

Let $W=Y^{t_1}XY^{t_2}\ldots Y^{t_q}X$ as in the proof of Theorem \ref{theorem:stairstep}. Since $w$ now has a periodic orbit of period exactly $q$,  we get that any $b-$string starting on adjacent $X'$s must land in adjacent $Y^*$ strings. Thus the constraints of the linear programming problem are invariant under permutation of the variable $t_i$,  and by convexity,  extrema is achieved when all $t_i$'s are equal. But then we get 
\[q.t_i=1\Rightarrow t_i=\vf{q}\] and
\[\beta_iu\geq l_i.t_i=\frac{(q\beta_i-1)}{q} \Rightarrow u\geq 1-\vf{q\beta_i} \qquad\forall\,  1\leq i\leq nq\]
Hence the minimum $u$ which gives a solution to the system of equation is given by 
\[u=1-\vf{q.\max_{1\leq i\leq n}\{\beta_i\}}.\]
Observing that equality is indeed achieved in case of the word $\left(XY^{\max\{\beta_i\}}\right)^q$,  we get equality in the second part of Theorem \ref{theorem:sigma_inequality}.

Again,  we can give a much simpler proof of this result using the algorithm in the last section. In this case,  we have $g=1$ so that $q=q'$. So $S$ is the maximum of all the $\beta_i$'s which correspond to any $A_i$ which is a translate of $A_1$ by one of $-h_a, -h_a+1, \ldots, 0, \ldots, h_a-1, h_a$; i.e. all of the $A_i$'s.  Thus $S=\sigma_w(g)=\max_{1\leq i\leq n}\{\beta_i\}$ since $g=1$.
\end{proof}

\begin{corollary}\label{corollary:sigma_prime}
If $h_a$ is a prime number then 
\[\fr_w(p/q)=\begin{cases} \displaystyle\frac{h_a}{q\cdot h_b},  & \text{if } h_a\mid q \\ \displaystyle\vf{q\cdot \max\limits_{1\leq i\leq n}\beta_i},  & \text{if } h_a \nmid q \end{cases}.\]
\end{corollary}
\begin{remark}
The function $\sigma_w(g)$ depends on $g=gcd(h_a,q)$ in a  complicated way when $h_a$ is not prime as we can see from the following table:
\begin{table}[H]
\centering
\begin{tabular}{|c||c|c|c|c|c|}
	\hline Word & &$p/q=1/5$ & $p/q=1/2$ & $p/q=1/3$ & $p/q=1/6$ \\
	\hline $h_a=6$ & $h_b$ &  $g=1$ & $g=2$ & $g=3$ & $g=6$ \\
	\hline\hline $aaabaaabbbb$& 5 & $4$ & $5/2$ & $4/3$ & $5/6$ \\
	\hline $abaabaaabbbb$ & 6 & 4 & 5/2 & 5/3 & 1 \\
	\hline $abbaabaaabbbb$ & 7 & 4 & 3 & 2 & 7/6 \\
	\hline $abbbaabaaabbbb$ & 8 &  4 & 7/2 & 4/3 & 7/3 \\
	\hline $abbbababaaabbbb$ & 9 & 4 & 7/2 & 8/3 & 3/2 \\
	\hline $abbbaabbaaabbbb$ & 9 & 4 & 7/3 & 7/3 & 3/2 \\
	\hline $abbbababbaaabbbb$ & 10 & 4 & 7/2 & 8/3 & 5/3 \\
	\hline 
\end{tabular} 
\caption{Values of $\sigma_w(g)$ for different $w$ and $g$}
\label{tab:sigma_values}
\end{table}

\end{remark}

\begin{example}\label{example:abaab_fringe_formula}
Let us consider the case of the word $w=abaab$. By corollary \ref{corollary:sigma_prime},  the left fringe lengths are given by\[
\fr_w(p/q)=\begin{cases}
\displaystyle\frac{3}{2q} & \text{ when } 3\mid q\\
\displaystyle\vf{q} & \text{ when } 3\nmid q
\end{cases}
\]
and the right fringe lengths are given by
\[\fr_w(p/q)=\begin{cases}
\displaystyle\frac{2}{3q} & \text{ when } q \text{ is even}.\\
\displaystyle\vf{2q} & \text{ when } q \text{ is odd}
\end{cases}.\]
The cases when $3\nmid q$ and $2\nmid q$ were also discussed in \cite{Calegari_Walker}, p $18$.

We finish this section by giving a Fringe plot for both sides for the word $w=abaab$. Let us put the origin at the point $(r=1, s=1)$ and the point $(r=0, s=0)$ be depicted as $(1, 1)$. Then we have the following picture.
\begin{figure}[H]
	\centering
	\includegraphics[trim = 6.5cm 2cm 4cm 1cm, clip, width=\textwidth]{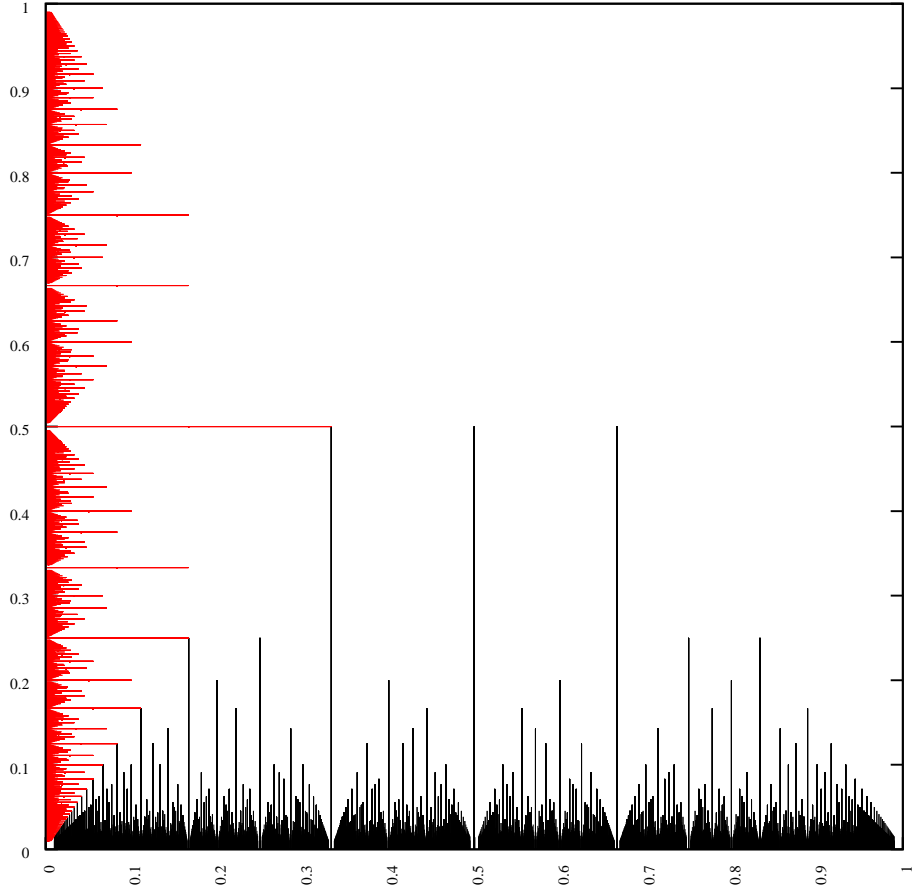}
    \caption{Plot of the fringes of $abaab$, $q=1$ to $100$}
\end{figure}

\end{example}

\section{Projective Self Similarity}\label{section:Projective_Self_Similarity}
In her paper \cite{Gordenko},  A. Gordenko shows that the the Ziggurat of the word $w=ab$ is self similar under two projective transformation (Theorem $4$). In this section we show that similar transformations exist in case of the word $w=abaab$,  which gives a different way to look at the Fringe formula.

Let us first look at the self-similarities of the Left Fringe. Below is a plot of the Fringe lengths where $x$-axis is the value of $\rot^\sim(a)$ and $y-$axis is value of $\fr_{abaab}(x)$. Thus for $x=p/q$ we have $\fr_{abaab}(x)$ defined as in Example (\ref{example:abaab_fringe_formula}). We will drop the subscript $abaab$ for the next part.

We prove that the unit interval can be decomposed into some finite number of intervals $\Delta_i$ such that there exist a further decomposition of each $\Delta_i$ into a disjoint union of subintervals $I_{i, j}$ such that the graph of $\fr(x)$ on each of $I_{i, j}$ is similar to that on some $\Delta_{k(i, j)}$ under  projective linear transformations as follows: 

\begin{figure}[H]
	\centering
	\includegraphics[trim = 3.5cm 2cm 1cm 0cm, clip,width=\linewidth]{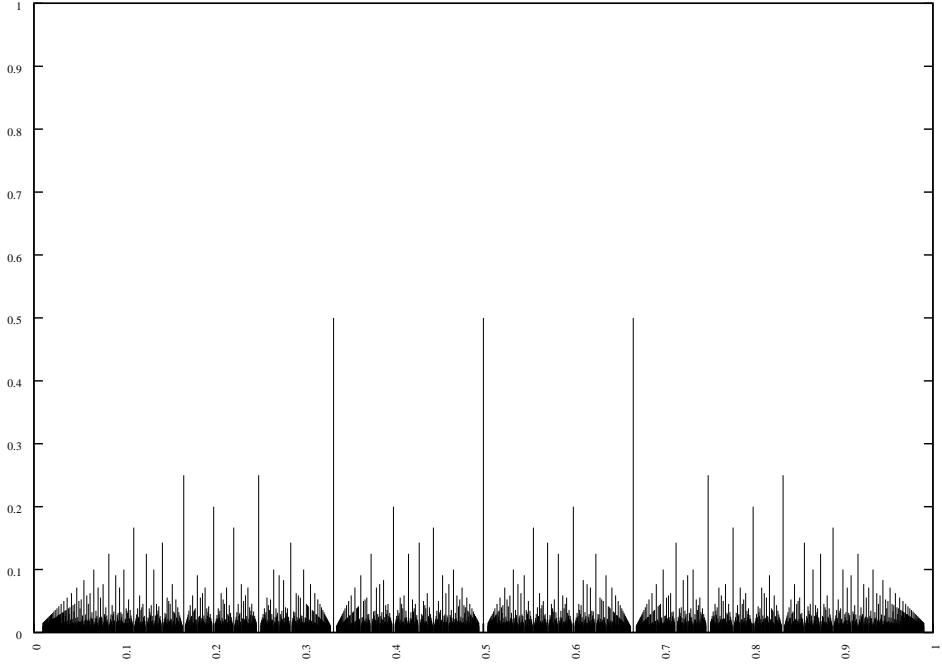}
	\caption{Plot of Left Fringe,  $q=1$ to $100$}
	\label{fig:left_fringe}
\end{figure}

\begin{theorem}\label{theorem:abaab_self_similarity}
Let $\Delta_1=(0, 1/3), \Delta_2=(1/3, 1/2),  \Delta_3=(1/2, 2/3)$ and $\Delta_4=(2/3, 1)$. Then we have the following decomposition into $I_{i, j}$ and transformations $T_{i, j}$:
\begin{align*}
&I_{1, 1}=(0, 1/4),  & T_{1, 1}(I_{1, 1})&=\Delta_1\cup\Delta_2\cup\Delta_3\cup\Delta_4=[0, 1], \\ 
&&&T_{1, 1}(x, y)=\left(\frac{x}{1-3x}, \frac{y}{1-3x}\right) \\
&I_{1, 2}=(1/4, 1/3),  & T_{1, 2}(I_{1, 2})&=\Delta_1, \\ &&&T_{1, 2}(x, y)=\left(\frac{4x-1}{9x-2}, \frac{y}{9x-2}\right) \\
&I_{2, 1}=(1/3, 1/2),  & T_{2, 1}(I_{2, 1})&=\Delta_1, \\ &&&T_{2, 1}(x, y)=\left(\frac{1-2x}{2-3x}, \frac{y}{2-3x}\right)
\end{align*}
Since the graph is clearly symmetric about $x=1/2$,  similar decomposition exists for $\Delta_3$ and $\Delta_4$ (See Figure \ref{fig:self_similarity_abaab}).
\end{theorem}

\begin{figure}[H]
	\def\svgwidth{\textwidth}
	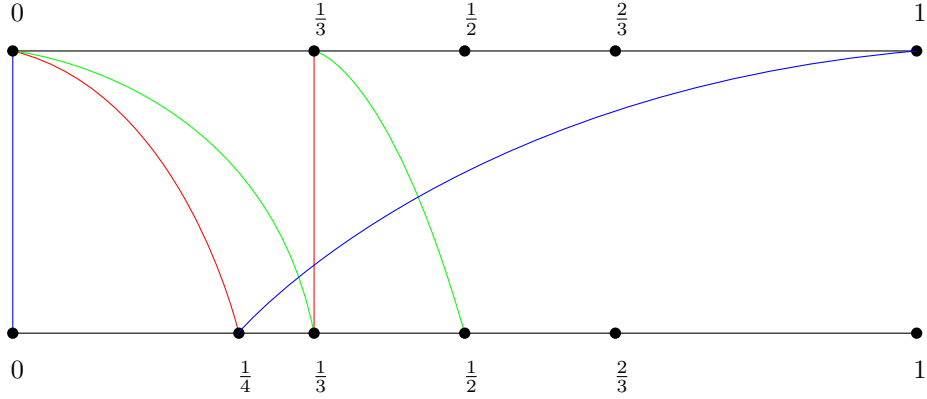
	\caption{Intervals of projective self similarity in case of $w=abaab$}
	\label{fig:self_similarity_abaab}
\end{figure}

\begin{proof}
 For each of the transformations note that the denominator of the image of $p/q$ has the same gcd with $h_a$ as $q$. Also,  in each case,  the numerator and denominator are coprime. The proof then follows easily by checking the length of images in each case.
\end{proof}

 We thus note that in fact $\Delta_1$ contains all the  information necessary to determine the fringe dynamics. In fact,  for $h_a$ prime the following similarity result always holds: 
\begin{theorem}
Let $\Delta_1=(0, 1/h_a)$ where $h_a$ is a prime number. Then we can decompose $\Delta_1$ into $I_{i, j}$ and find transformations $T_{i, j}$ as follows:
\begin{align*}
&I_{1, 1}=(0, 1/(h_a+1)),  & T_{1, 1}(I_{1, 1})&=[0, 1], \\ &&T_{1, 1}(x, y)&=\left(\frac{x}{1-h_ax}, \frac{y}{1-h_ax}\right) \\
&I_{1, 2}=(1/(h_a+1), 1/h_a),  & T_{1, 2}(I_{1, 2})&=\Delta_1, \\ &&T_{1, 2}(x, y)&=\left(\frac{(h_a+1)x-1}{h_a^2x-(h_a-1)}, \frac{y}{h_a^2x-(h_a-1)}\right)
\end{align*}
\end{theorem}

It is also easy to prove in the case of prime $h_a$ that  the plot on $\Delta=[\frac{(h_a-1)}{2h_a}, \vf{2}]$ is similar to $\Delta_1$ under the transformation\[T(x, y)=\left(\frac{2-4x}{(h_a+1)-2h_ax}, \frac{2y}{(h_a+1)-2h_ax}\right)\]
Note that in case of $h_a=3$,  we have $(h_a-1)/2h_a=1/h_a$,  which explains Theorem \ref{theorem:abaab_self_similarity}.

\end{document}